\documentclass[11pt, reqno, psamsfonts]{amsart}
\pdfoutput=1

\usepackage{amssymb}
\usepackage{amsthm}
\usepackage{amsmath}
\usepackage{latexsym}
\usepackage[T1]{fontenc}
\usepackage[utf8]{inputenc}
\usepackage[russian, french, english]{babel}

\usepackage{graphicx}
\usepackage{wrapfig}
\usepackage{mathtools}
\usepackage{tikz}
\usepackage{amsbsy}
\usepackage[inline]{enumitem}
\usepackage{mathrsfs}
\usetikzlibrary{shapes,snakes}
\usetikzlibrary{arrows.meta}
\usetikzlibrary{decorations.pathmorphing}
\usetikzlibrary{patterns}
\usetikzlibrary{calc}
\usepackage{float}
\usepackage{array}
\usepackage{subcaption}
\usepackage{multicol}
\usepackage{stmaryrd}
\usepackage{cancel} %to strike through text
\usepackage{algorithm}
\usepackage{algpseudocode}
\makeatletter
\algnewcommand{\LineComment}[1]{\Statex \hskip\ALG@thistlm {\color{gray}\texttt{// #1}}}
\makeatother

\counterwithin{algorithm}{section}
\usepackage[justification=centering, labelfont=bf]{caption}
\usepackage{lmodern}
\usepackage{mathabx}

\usepackage{upgreek}
\usepackage{titlesec}
\usepackage{bbm}
\usepackage[spacing=true,kerning=true,babel=true,tracking=true]{microtype}
\usepackage[shortcuts]{extdash}
\usepackage[foot]{amsaddr}
\usepackage[left=1in,right=1in,top=1in,bottom=1in,bindingoffset=0cm]{geometry}
\usepackage{bm}
\usepackage{centernot}
\usepackage{mdframed}
\usepackage{hyperref}
\usepackage{mathdots}
\usepackage{xspace}
\hypersetup{
    colorlinks=true,%change to true to get colors
    linkcolor=blue,
    filecolor=magenta,      
    urlcolor=red,
    citecolor=gray,
}
\allowdisplaybreaks

\usepackage[backend=biber, style=alphabetic, sorting=nyt, maxnames=100,backref=true, firstinits=true, sortcites = true]{biblatex}
\title{Balanced independent sets and colorings of hypergraphs}
\date{}
\author{\lsstyle Abhishek~Dhawan}
\email{abhishek.dhawan@math.gatech.edu}
\address{\textls{\normalfont{}School of Mathematics, Georgia Institute of Technology, Atlanta, GA, USA}}

\newtheoremstyle{bfnote}%
{}{}%
{\slshape}{}%
{\bfseries}{\bfseries.}%
{ }%
{\thmname{#1}\thmnumber{ #2}\thmnote{ \ep{\normalfont{}#3}}}

\theoremstyle{bfnote}
\newtheorem{theo}[equation]{Theorem}
\newtheorem*{theo*}{Theorem}

\newtheorem{Lemma}[equation]{Lemma}
\newtheorem{claim}{Claim}[equation]

\newtheorem{fact}[equation]{Fact}

\newtheorem*{corl*}{Corollary}

\newcounter{ForClaims}[section]

\theoremstyle{definition}
\newtheorem{defn}[equation]{Definition}
\newtheorem*{defn*}{Definition}

\newtheorem*{exmp*}{Example}

\theoremstyle{remark}
\newtheorem*{ques*}{Question}
\newtheorem*{remk*}{Remark}

\newcommand*{\myproofname}{Proof}
\newenvironment{claimproof}[1][\myproofname]{\begin{proof}[#1]}{\end{proof}}

\newcommand{\0}{\emptyset}
\newcommand{\set}[1]{\{#1\}}
\newcommand{\N}{{\mathbb{N}}}

\newcommand{\R}{\mathbb{R}}

\renewcommand{\P}{\mathbb{P}}
\newcommand{\E}{\mathbb{E}}
\renewcommand{\epsilon}{\varepsilon}
\newcommand{\eps}{\epsilon}
\renewcommand{\phi}{\varphi}
\renewcommand{\theta}{\vartheta}
\renewcommand{\leq}{\leqslant}
\renewcommand{\geq}{\geqslant}
\newcommand{\defeq}{\coloneqq}

\newcommand{\bemph}[1]{{\normalfont#1}} % define how emphasised brackets should look
\newcommand{\ep}[1]{\bemph{(}#1\bemph{)}} % parentheses

\newcommand{\emphdef}[1]{\textbf{\textit{{#1}}}}

\newcommand{\emphd}[1]{\emphdef{#1}}

\numberwithin{equation}{section}

\newcommand{\bbone}{\mathbbm{1}}

\titleformat{\subsection}[block]{\bfseries}{\thesubsection.}{1ex}{}
\titleformat{\subsubsection}[runin]{\itshape}{\bfseries\upshape\thesubsubsection.}{1ex}{}[.---]

\titleformat{\section}[block]{\scshape\filcenter}{\thesection.}{1ex}{}
\titlespacing*{\section}{0pt}{*3}{*1}
\titlespacing*{\subsection}{0pt}{*3}{*1}
\titlespacing*{\subsubsection}{0pt}{*1.5}{*0}

\renewbibmacro{in:}{}

\renewbibmacro*{volume+number+eid}{%
	\printfield{volume}%
	\setunit*{\addnbspace}% NEW (optional); there's also \addnbthinspace
	\printfield{number}%
	\setunit{\addcomma\space}%
	\printfield{eid}}
\DeclareFieldFormat[article]{volume}{\textbf{#1}\space}
\DeclareFieldFormat[article]{number}{\mkbibparens{#1}}

\DeclareFieldFormat{journaltitle}{#1,}
\DeclareFieldFormat[thesis]{title}{\mkbibemph{#1}\addperiod}
\DeclareFieldFormat[article, unpublished, thesis]{title}{\mkbibemph{#1},}
\DeclareFieldFormat[book]{title}{\mkbibemph{#1}\addperiod}
\DeclareFieldFormat[unpublished]{howpublished}{#1, }

\DeclareFieldFormat{pages}{#1}

\DeclareFieldFormat[article]{series}{Ser.~#1\addcomma}

\setlist{topsep=3pt,itemsep=3pt}

\begin{document}

\vspace*{0pt}

\maketitle
\begin{abstract}
    A $k$-uniform hypergraph $H = (V, E)$ is $k$-partite if $V$ can be partitioned into $k$ sets $V_1, \ldots, V_k$ such that every edge in $E$ contains precisely one vertex from each $V_i$.
    We call such a graph $n$-balanced if $|V_i| = n$ for each $i$.
    An independent set $I$ in $H$ is balanced if $|I\cap V_i| = |I\cap V_j|$ for each $1 \leq i, j \leq k$, and a coloring is balanced if each color class induces a balanced independent set in $H$.
    In this paper, we provide a lower bound on the balanced independence number $\alpha_b(H)$ in terms of the average degree $D = |E|/n$, and an upper bound on the balanced chromatic number $\chi_b(H)$ in terms of the maximum degree $\Delta$.
    Our results recover those of recent work of Chakraborti for $k = 2$.
\end{abstract}

\noindent

\section{Introduction}\label{section:intro}

All hypergraphs considered are finite and undirected.
A hypergraph is a generalization of a graph in which an edge can contain more than two vertices.
If every edge contains exactly $k$ vertices, we say the hypergraph is $k$-uniform (graphs are $2$-uniform hypergraphs).
Hypergraphs find a wide array of applications in satisfiability problems, Steiner triple systems, and particle tracking in physics, to name a few.
Before we state our results, we make a few definitions.
For $q \in \N$, we let $[q] \defeq \set{1, \ldots, q}$.
For a set $X$ and $k \in \N$, we let $\binom{X}{k}$ be the collection of $k$-element subsets of $X$, and we let $2^X$ denote the power set of $X$.
Let $H = (V, E)$ be an undirected $k$-uniform hypergraph, i.e., $E \subseteq \binom{V}{k}$.
We say $H$ is $k$-partite if there is a partition $V_1\cup \cdots\cup V_k$ of $V$ such that each edge $e \in E$ satisfies $|e\cap V_i| = 1$ for each $i$.
Furthermore, we say such a hypergraph is $n$-balanced if $|V_i| = n$ for each $i$.
For each $v \in V$, we let $E_H(v)$ denote the edges containing $v$, $N_H(v)$ denote the set of vertices contained in the edges in $E_H(v)$ apart from $v$ itself, $\deg_H(v) \defeq |E_H(v)|$, and $\Delta(H) \defeq \max_{u\in V}\deg_H(u)$.
In this paper, we are interested in independent sets and colorings of hypergraphs.
An independent set is a set $I \subseteq V$ containing no edges, and a proper coloring is a partition of $V$ into independent sets.

When considering $k$-partite hypergraphs, $\bigcup_{i \in J}V_i$ is an independent set for any $J \subsetneq [k]$, and a trivial $2$-coloring exists.
Namely, color the vertices in $V_1$ blue, and color all others red.
For $k=2$, bipartite graphs have been extensively studied.
There are a number of graph-theoretic problems where the interest is solely on results for bipartite graphs.
As $k$-partite hypergraphs are a generalization of bipartite graphs, it is natural to consider extensions of such problems to this setting.
We will consider the so-called \textit{balanced independence number} and \textit{balanced chromatic number}.

\begin{defn}\label{def:independent}
    Let $H = (V_1 \cup \cdots \cup V_k, E)$ be a $k$-uniform $k$-partite hypergraph.
    An independent set $I$ in $H$ is \emphd{balanced} if $|I\cap V_i| = |I\cap V_j|$ for each $i, j \in [k]$.
    For the largest balanced independent set $I \subseteq V(H)$, we let $\alpha_b(H) \defeq |I|/k$ denote the \emphd{balanced independence number}.
\end{defn}

\begin{defn}\label{def:color}
    Let $H = (V_1 \cup \cdots \cup V_k, E)$ be an $n$-balanced $k$-uniform $k$-partite hypergraph.
    A proper $q$-coloring $\phi\,:\,V(H) \to [q]$ is \emphd{balanced} if the color classes $\phi^{-1}(i)$ are balanced independent sets for each $i \in [q]$.
    The \emphd{balanced chromatic number} (denoted $\chi_b(H)$) is the minimum number of colors required for a balanced coloring.
\end{defn}

Balanced independent sets for bipartite graphs have gained recent interest from both a theoretical standpoint \cite{axenovich2021bipartite, chakraborti2023extremal}, as well as an algorithmic one \cite{perkins2024hardness}.
They were first introduced in \cite{ash1983two}, where the author aimed to determine when a bipartite graph contains a Hamilton cycle.
There has been a lot of work regarding lower bounds for $\alpha_b(G)$ in terms of the average (or maximum) degree of $G$.
In \cite{favaron1993bipartite}, the authors provided a general bound for all graphs as well as for some specific classes including sparse graphs and forests. 
Axenovich, Sereni, Snyder and Weber determined the correct asymptotic order of $\alpha_b(G)$ showing the following result:

\begin{theo}[\cite{axenovich2021bipartite}]\label{theo:old}
    For all $\eps > 0$, the following holds for $\Delta \in \N$ sufficiently large. 
    % We have the following:
    \begin{enumerate}[label=\ep{\normalfont{}\arabic*}]
        \item\label{item:upper-bound-graph} For $n \geq \Delta$, there exists an $n$-balanced bipartite graph $G$ of maximum degree at most $\Delta$ such that
        $\alpha_b(G) \leq (2+\eps)n\log\Delta/\Delta$.
        \item\label{item:lower-bound-graph} For $n \geq 5\Delta\log\Delta$, let $G$ be an $n$-balanced bipartite graph of maximum degree $\Delta$.
        Then, $\alpha_b(G) \geq n\log\Delta/(2\Delta)$.
    \end{enumerate}
\end{theo}

The result in Theorem~\ref{theo:old}\ref{item:lower-bound-graph} was recently improved by Chakraborti, where they improved the constant in the bound as well as extended the result to depend on the average degree rather than the maximum.
The relationship between $n$ and $D$ was not stated explicitly, but can be inferred from their arguments.

\begin{theo}[\cite{chakraborti2023extremal}]\label{theo:Chakraborti}
    For all $\eps > 0$, the following holds for $D \in \N$ sufficiently large and some constant $C\defeq C(\eps) > 0$.
    For $n \geq CD^2/\log D$, let $G$ be an $n$-balanced bipartite graph containing $Dn$ edges.
    Then, $\alpha_b(G) \geq (1-\epsilon)n\log D/D$.
\end{theo}

To the best of our knowledge, this is the first paper to consider the extension to hypergraphs.
Our first result provides a bound on the asymptotic order of $\alpha_b(H)$, matching the results in Theorem~\ref{theo:old}\ref{item:upper-bound-graph} and Theorem~\ref{theo:Chakraborti} for $k = 2$.

\begin{theo}\label{theo:independent}
    For all $\eps > 0$ and $k \geq 2$, the following holds for $\Delta, D \in \N$ sufficiently large and some constant $C\defeq C(\eps, k) > 0$. 
    % We have the following:
    \begin{enumerate}[label=\ep{\normalfont{I}\arabic*}]
        \item\label{item:upper-bound} For $n \geq \Delta^{1/(k-1)}$, there exists an $n$-balanced $k$-uniform $k$-partite hypergraph $H$ of maximum degree at most $\Delta$ such that
        \[\alpha_b(H) \leq \left(\left(\frac{k+\eps}{k-1}\right)\frac{\log \Delta}{\Delta}\right)^{1/(k-1)}n.\]
        \item\label{item:lower-bound} For $n \geq C\,D^{1/(k-1)}(\log D)^{1 - 1/(k-1)}$, let $H$ be an $n$-balanced $k$-uniform $k$-partite hypergraph containing $Dn$ edges.
        Then,
        \[\alpha_b(H) \geq \left(\left(\frac{1-\eps}{k-1}\right)\frac{\log D}{D}\right)^{1/(k-1)}n.\]
    \end{enumerate}
\end{theo}

We note that our bound on $n$ in \ref{item:upper-bound} matches the one in Theorem~\ref{theo:old}\ref{item:upper-bound-graph} for $k = 2$.
Furthermore, our bound on $n$ in \ref{item:lower-bound} improves upon the one in Theorem~\ref{theo:Chakraborti} for $k = 2$.
The proof of Theorem~\ref{theo:independent}\ref{item:upper-bound} involves considering the Erd\H{o}s--R\'enyi $k$-uniform $k$-partite hypergraph $\mathcal{H}(k, n, p)$.
Here, each partition has size $n$ and each valid edge is included independently with probability $p$.
We show that for an appropriate choice of $p$, an instance of $\mathcal{H}(k, n, p)$ does not contain a large balanced independent set with high probability.
For Theorem~\ref{theo:independent}\ref{item:lower-bound}, we apply a randomized procedure to construct a large balanced independent set.
The approach is inspired by that of \cite{chakraborti2023extremal} for graphs, however, due to the additional structural constraints of hypergraphs, the arguments are much more technical, requiring additional tools described in \S\ref{sec:prelim}.
Furthermore, the use of certain tools (specifically the \hyperref[theo:markov]{Reverse Markov inequality}) allow us to improve the bound on $n$, as mentioned earlier.

Our next result concerns balanced colorings (see Definition~\ref{def:color}).
For bipartite graphs, Feige and Kogan first determined the asymptotic order of $\chi_b(G)$, showing $\chi_b(G) \leq 20\Delta/(\eps^2\log\Delta)$ \cite{feige2010balanced}.
(A lower bound of $\Delta/((2+\eps)\log \Delta)$ follows from Theorem~\ref{theo:old}\ref{item:upper-bound-graph}.)
The constant factor was improved by Chakraborti, who showed the following:

\begin{theo}[\cite{chakraborti2023extremal}]\label{theo:Chakraborti_color}
    For all $\eps > 0$, the following holds for $\Delta\in \N$ sufficiently large and some constant $C\defeq C(\eps) > 0$.
    For $n \geq C\,\Delta^2\log\Delta$, let $G$ be an $n$-balanced bipartite graph of maximum degree $\Delta$.
    Then, $\chi_b(G) \leq (1+\eps)\Delta/\log\Delta$.
\end{theo}

Once again, the relationship between $n$ and $\Delta$ was not stated explicitly stated in their paper.
This is the first paper to consider balanced colorings of hypergraphs.
Our result provides a bound on the asymptotic order of $\chi_b(H)$, matching the above result for $k = 2$.

\begin{theo}\label{theo:color}
    For all $\eps > 0$ and $k \geq 2$, the following holds for $\Delta \in \N$ sufficiently large and some constant $C\defeq C(\eps, k) > 0$.
    For $n \geq C\Delta^{5 + 1/(k-1)}\log\Delta$, let $H$ be an $n$-balanced $k$-uniform $k$-partite hypergraph of maximum degree $\Delta$.
    Then, 
    \[\chi_{b}(H) \leq \left(\left(k - 1 + \eps\right)\frac{\Delta}{\log \Delta}\right)^{1/(k-1)}.\]
\end{theo}

We remark that this result is sharp up to a factor of $k^{1/(k-1)}$ (as a result of Theorem~\ref{theo:independent}\ref{item:upper-bound}).
The proof is constructive, following a similar argument to that of Theorem~\ref{theo:Chakraborti_color}.
In particular, we apply a randomized procedure to construct a ``good'' partial $q$-coloring using $(1-\eps)\,q$ colors such that the subhypergraph induced by the uncolored vertices can be colored with the remaining $\eps\,q$ colors.
We remark that the use of certain probabilistic tools (namely, \hyperref[theo:harris]{Harris's inequality} and \hyperref[theo:Talagrand]{Talagrand's inequality}) allows for a much simpler analysis of the procedure (even for $k = 2$).

A curious feature of Theorem~\ref{theo:color} is that the bound matches known results on the \textit{(list) chromatic numbers} of other hypergraphs.
Frieze and Mubayi showed that $\chi(H) = O\left(\left(\Delta(H)/\log\Delta(H)\right)^{1/(k-1)}\right)$ for \textit{simple hypergraphs}, i.e., where every pair of vertices in $H$ are contained in at most one common edge \cite{frieze2013coloring}.
Iliopoulos proved an analogous bound for hypergraphs of girth at least $5$ \cite{iliopoulos2021improved}.
For $k = 3$, Cooper and Mubayi extended this result to triangle-free hypergraphs \cite{cooper2016coloring}.
A triangle in a hypergraph is a set of three pairwise intersecting edges with no common vertex.
This result was recently extended to include all $k$-uniform triangle-free hypergraphs for $k \geq 3$ by Li and Postle \cite{li2022chromatic}.

We conclude this section with a discussion of potential future directions of inquiry.
First, we note that the lower bounds on $n$ in the results of Theorem~\ref{theo:independent} match those known for the $k = 2$ case, however, the lower bound on $n$ is far from optimal in Theorem~\ref{theo:color}.
In particular, the latter result only holds for sparse hypergraphs (as $\Delta$ can be as large as $n^{k-1}$ in theory).

\begin{ques*}
    Can we get a similar bound on $\chi_b(H)$ for hypergraphs satisfying $n \approx \Delta^{\Theta(1/(k-1))}$?
\end{ques*}

Chakraborti discusses the relationship of Theorem~\ref{theo:Chakraborti_color} to the Johansson--Molloy theorem \cite{Joh_triangle, Molloy}, which provides a similar bound for the list chromatic number of triangle-free graphs.
In general, results on bipartite graphs tend to hold for triangle-free graphs as most triangle-free graphs are bipartite (and all bipartite graphs are triangle-free).
In the hypergraph case, however, $k$-partite hypergraphs need not be triangle-free!
This makes our result somewhat surprising as it matches the bound for the list chromatic number of triangle-free hypergraphs as mentioned earlier.
In subsequent work, the author shows how the tools in this paper apply for list colorings of $k$-partite hypergraphs \cite{corr}.

Another avenue for research is to determine the computational hardness of finding large balanced independent sets.
In \cite{perkins2024hardness}, Perkins and Wang determined that no \textit{local} or \textit{low-degree} algorithm can find a balanced independent with at least $(1+\eps)n\log d/d$ vertices in each partition of $\mathcal{H}(2, n, d/n)$ with high probability.

\begin{ques*}
    Are there local or low-degree algorithms that can find a balanced independent set with at least $n\left(\frac{(1+\eps)}{(k-1)}\,\log d/d)\right)^{1/(k-1)}$ vertices in each partition of $\mathcal{H}(k, n, d/n^{k-1})$ with high probability?
\end{ques*}

The techniques employed in this paper are similar to those for other combinatorial problems related to bipartite graphs.
It would be worth investigating when results on bipartite graphs extend to the $k$-partite setting.

The rest of the paper is structured as follows.
In \S\ref{sec:prelim}, we will prove some preliminary facts, as well as describe the probabilistic tools we will employ.
In \S\ref{sec:indep}, we will prove Theorem~\ref{theo:independent}, and in \S\ref{sec:color}, we will prove Theorem~\ref{theo:color}.

\section{Preliminaries}\label{sec:prelim}

For a $k$-uniform $k$-partite hypergraph $H = (V_1\cup\cdots\cup V_k, E)$, the $k$-partite complement of $H$ is the hypergraph $H^c = (V_1\cup\cdots\cup V_k, E')$, where $E'$ contains all valid edges not in $E$ (an edge is valid if it has exactly one endpoint in each partition $V_i$).
For any $S\subseteq V(H)$ such that $|S\cap V_i| \leq 1$ for each $i \in [k]$, we let $\deg_H(S)$ be the number of edges in $H$ containing $S$, and we define $\delta_j(H) \defeq \min_{|S| = j}\deg_H(S)$ (where we only consider $S$ such that $|S\cap V_i| \leq 1$ for each $i \in [k]$).
Finally, recall that a matching in a hypergraph is a set of pairwise disjoint edges.
A perfect matching $M$ in $H$ is a matching such that $v$ is contained in some edge in $M$ for every $v \in V(H)$.
We begin with the following fact for hypergraphs, which is identical to an observation made for bipartite graphs in \cite{feige2010balanced}.

\begin{fact}\label{fact:coloring_matching}
    A $k$-uniform $k$-partite hypergraph $H$ has a balanced coloring if and only if its $k$-partite complement has a perfect matching.
\end{fact}

\begin{claimproof}
    For the forward implication, we can find a perfect matching in each color class (as it is a balanced independent set).
    This forms a perfect matching in $H^c$.
    For the backward implication, assign matched vertices the same color.
\end{claimproof}

As a consequence, we have the following lemma.

\begin{Lemma}\label{lemma:rDelta}
    If $H$ is an $n$-balanced $k$-partite hypergraph of maximum degree $\Delta \leq n/2$, then $\chi_b(H) \leq k\Delta + 1$.
\end{Lemma}

\begin{proof}
    Consider the $k$-partite complement $H^c$ of $H$.
    Note that $\delta_{k-1}(H^c) \geq n - \Delta \geq n/2$.
    It follows from \cite[Theorem 2]{aharoni2009perfect} that $H^c$ contains a perfect matching $M \defeq \set{e_1, \ldots, e_n}$.
    Let us color the vertices of $H$ such that for each $i \in [n]$ the vertices in $e_i$ receive the same color.
    Clearly this is a balanced coloring.
    When coloring the vertices in $e_i$, we must ensure the resulting coloring is proper.
    To this end, we note:
    \[|\set{e \in E(H)\,:\, e\cap e_i \neq \0}| \leq k\Delta.\]
    In particular, there is always at least one available color for the vertices in $e_i$, as desired.
\end{proof}

In our proofs, we will employ three concentration tools.
The first of these is Markov's inequality.

\begin{theo}[{Markov; \cite[\S3]{MolloyReed}}]\label{theo:markov}
    Let $X$ be a non-negative random variable. 
    Then, 
    \[\P[X \geq a] \leq \frac{\E[X]}{a}.\]
    Furthermore, if $X \leq t$, then
    \[\P[X \leq a] \leq \frac{t - \E[X]}{t - a}.\]
\end{theo}

The second statement above is often referred to as the \emphd{Reverse Markov inequality}.
The next tool we will use is the Chernoff Bound for binomial random variables. 
We state the two-tailed version below:

\begin{theo}[{Chernoff; \cite[\S5]{MolloyReed}}]\label{theo:chernoff}
    Let $X$ be a binomial random variable on $n$ trials with each trial having probability $p$ of success. Then for any $0 \leq \xi \leq \E[X]$, we have
    \begin{align*}
        \P\Big[\big|X - \E[X]\big| \geq \xi\Big] < 2\exp{\left(-\frac{\xi^2}{3\E[X]}\right)}. 
    \end{align*}
\end{theo}

We will also take advantage of Talagrand's inequality.
The original version in \cite[\S10.1]{MolloyReed} contained an error which was rectified in a subsequent paper of the same authors. 
We state the version from \cite{molloy2014colouring} here.

\begin{theo}[{Talagrand's Inequality; \cite{molloy2014colouring}}]\label{theo:Talagrand}
    Let $X$ be a non-negative random variable, not identically $0$, which is a function of $n$ independent trials $T_1$, \ldots, $T_n$. Suppose that $X$ satisfies the following for some $\gamma$, $r > 0$: 
    \begin{enumerate}[label=\ep{\normalfont{}T\arabic*}]
        \item Changing the outcome of any one trial $T_i$ can change $X$ by at most $\gamma$.
        \item For any $s>0$, if $X \geq s$ then there is a set of at most $rs$ trials that certify $X$ is at least $s$.
    \end{enumerate}
    Then for any $\xi \geq 0$, we have
    \begin{align*}
        \P\Big[\big|X-\E[X]\big| \geq \xi + 20\gamma\sqrt{r\E[X]} + 64\gamma^2r\Big]\leq 4\exp{\left(-\frac{\xi^2}{8\gamma^2r(\E[X] + \xi)}\right)}.
    \end{align*}
\end{theo}

We will also need the FKG inequality \cite{fortuin1971correlation} along with a special case of it, dating back to Harris \cite{Harris} and Kleitman \cite{Kleitman}.
Before we state the inequalities, we make a few definitions.
Let $X$ be a set, $f, g\,:\, 2^X \to \R$, and $\mathcal{A} \subseteq 2^X$ be such that
\begin{itemize}
    \item $f(S) \leq f(T)$ whenever $S \subseteq T$, 
    \item $g(S) \geq g(T)$ whenever $S \subseteq T$, and
    \item $T\in \mathcal{A} \implies S \in \mathcal{A}$ whenever $S \subseteq T$.
\end{itemize}
We say that $f$ is an \emphd{increasing function} on $2^X$, $g$ is a \emphd{decreasing function} on $2^X$, and $\mathcal{A}$ is a \emphd{decreasing family} of subsets of $X$.
The original version of the theorem below is stated with regards two decreasing families, however, as the intersection of decreasing families is decreasing, it can be shown that the inequality holds in the following more general form.

\begin{theo}[{Harris's inequality/Kleitman's Lemma \cite[\S6]{AlonSpencer}}]\label{theo:harris}
    Let $X$ be a finite set and let $f$ and $g$ be an increasing and decreasing function on $2^X$, respectively.
    Let $S \subseteq X$ be a random subset of $X$ obtained by selecting each $x \in X$ independently with probability $p_x \in [0,1]$. If $\mathcal{A}_1, \ldots, \mathcal{A}_n$ are decreasing families of subsets of $X$, then 
    \begin{itemize}
        \item (FKG Inequality) $\E[f(S)g(S)] \leq \E[f(S)]\E[g(S)]$, and
        \item (Harris's Inequality) $\P\left[S \in \bigcap_{i \in [n]}\mathcal{A}_i\right] \,\geq\, \prod_{i \in [n]}\P[S \in \mathcal{A}_i]$.
    \end{itemize}
\end{theo}

\section{Proof of Theorem~\ref{theo:independent}}\label{sec:indep}

We will split this section into two subsections, containing the proofs of \ref{item:upper-bound} and \ref{item:lower-bound}, respectively.

\subsection{Upper Bound}

Let us define the following parameters for $\gamma \defeq \eps/(2k^2)$:
\[N \defeq \frac{n}{1-\gamma}, \quad p \defeq \frac{\Delta}{(1+\gamma)N^{k-1}}, \quad s \defeq \left(\left(\frac{k+\eps}{k-1}\right)\frac{\log \Delta}{\Delta}\right)^{1/(k-1)}n.\]
Consider the Erd\H{o}s--R\'enyi $k$-partite hypergraph $H \sim \mathcal{H}(k,N,p)$.
We will show that with high probability $H$ contains an $n$-balanced subhypergraph of maximum degree at most $\Delta$ with no balanced independent set containing $s$ vertices in each partition.
Let $V_1, \ldots, V_k$ denote the partitions of $H$, and let us bound the probability that such a balanced independent set exists in $H$.

\begin{Lemma}\label{lemma:no_edges_many}
    $\P\left[\text{there is a balanced independent set $I$ in $H$ of size $sk$}\right] \leq \exp\left(-\frac{\eps\,s\log\Delta}{10k}\right)$.
\end{Lemma}

\begin{proof}
    Let $A_i \subseteq V_i$ be arbitrary sets of size $s$.
    The set $A\defeq \bigcup_{i \in [k]}A_i$ forms a balanced independent set if and only if there is no edge in the subhypergraph induced by $A$.
    In particular, we have
    \[\P\left[A \text{ is a balanced independent set}\right] = (1- p)^{s^k}.\]
    From here, we may conclude by a union bound:
    \begin{align*}
        \P\left[\text{there is a balanced independent set $I$ in $H$ of size $sk$}\right] &\leq \binom{N}{s}^k(1- p)^{s^k} \\
        &\leq \left(\frac{eN}{s}\right)^{sk}\exp\left(-ps^k\right) \\
        &= \exp\left(-s\left(ps^{k-1} - k\left(1 + \log\frac{N}{s}\right)\right)\right).
    \end{align*}
    Let us consider the exponent above.
    We have
    \begin{align*}
        ps^{k-1} &= \frac{\Delta}{(1+\gamma)N^{k-1}}\,\left(\left(\frac{k+\eps}{k-1}\right)\frac{\log \Delta}{\Delta}\right)\,n^{k-1} \\
        &= \left(\frac{k+\eps}{k-1}\right)\left(\frac{(1-\gamma)^{k-1}}{1+\gamma}\right)\log \Delta \\
        &\geq \left(\frac{k+\eps}{k-1}\right)(1 - k\gamma)\log\Delta.
    \end{align*}
    Furthermore,
    \begin{align*}
        \log\frac{N}{s} &= \log\left(\left(\left(\frac{k-1}{k+\eps}\right)\frac{\Delta}{\log \Delta}\right)^{1/(k-1)}\,\frac{1}{1-\gamma}\right) \\
        &\leq \frac{1}{k-1}\,\log\Delta + \log \left(\frac{1}{1-\gamma}\right).
    \end{align*}
    From these inequalities, we can conclude:
    \begin{align*}
        ps^{k-1} - k\left(1 + \log\frac{N}{s}\right) &\geq \left(\frac{\eps - k\gamma(k+\eps)}{k-1}\right)\log\Delta - k\left(1 + \log \left(\frac{1}{1-\gamma}\right)\right) \\
        &\geq \left(\frac{\eps}{10k}\right)\log \Delta,
    \end{align*}
    for $\Delta$ large enough and $\eps$ small enough.
    Plugging this value in above completes the proof.
\end{proof}

Next, let us show that very few vertices have large degree in $H$.

\begin{Lemma}\label{lemma:small_degree}
    For each $i \in [k]$, let $B_i \subseteq V_i$ consist of vertices having degree at least $\Delta$.
    Then, we have
    \[\P\left[\exists i \in [k],\, |B_i| \geq \gamma N\right] \leq k\,\exp\left(-\frac{\gamma^3N\Delta}{10(1+\gamma)}\right).\]
\end{Lemma}

\begin{proof}
    Let $v\in V_i$ be arbitrary.
    We will concentrate $\deg_H(v)$ using \hyperref[theo:chernoff]{Chernoff's inequality}.
    In particular, we have
    \[\P\left[|\deg_H(v) - \E[\deg_H(v)]|\geq \gamma\E[\deg_H(v)]\right] \leq 2\exp\left(-\frac{\gamma^2\E[\deg_H(v)]}{3}\right).\]
    Note that $\E[\deg_H(v)] = N^{k-1}p$.
    Furthermore, for $u, v \in V_i$, the random variables $\deg_H(v)$ and $\deg_H(u)$ are independent as $H$ is $k$-partite.
    Let $A \subseteq V_i$ be a set of size $\gamma \,N$.
    It follows that:
    \[\P\left[\forall v \in A,\, \deg_H(v) \geq \Delta\right] \,=\, \P\left[\forall v \in A,\, \deg_H(v) \geq (1+\gamma)N^{k-1}p\right] \,\leq\, \left(2\exp\left(-\frac{\gamma^2N^{k-1}p}{3}\right)\right)^{\gamma N}.\]
    In particular, we have
    \begin{align*}
        \P[|B_i| \geq \gamma N] &\leq \binom{N}{\gamma N}\left(2\exp\left(-\frac{\gamma^2N^{k-1}p}{3}\right)\right)^{\gamma N} \\
        &\leq \left(\frac{2e}{\gamma}\exp\left(-\frac{\gamma^2N^{k-1}p}{3}\right)\right)^{\gamma N} \\
        &= \exp\left(-\gamma\,N\left(\frac{\gamma^2N^{k-1}p}{3} - \log\frac{2e}{\gamma}\right)\right) \\
        &= \exp\left(-\gamma\,N\left(\frac{\gamma^2\Delta}{3(1+\gamma)} - \log\frac{2e}{\gamma}\right)\right) \\
        &\leq \exp\left(-\frac{\gamma^3N\Delta}{10(1+\gamma)}\right),
    \end{align*}
    for $\Delta$ large enough.
    The claim now follows by a union bound over $[k]$.
\end{proof}

The complements of the events in Lemmas~\ref{lemma:no_edges_many} and \ref{lemma:small_degree} occur with probability at least 
\[1 - \exp\left(-\frac{\eps\,s\log\Delta}{10k}\right) - k\,\exp\left(-\frac{\gamma^3N\Delta}{10(1+\gamma)}\right) > 0,\]
for $\Delta$ large enough.
Consider such an output $H$.
Note that a balanced independent set in any induced subhypergraph $H'$ of $H$ will be a balanced independent set in $H$ as well.
Consider the hypergraph $H'$ obtained from $H$ by removing the $\gamma\,N$ highest degree vertices from each $V_i$.
By Lemma~\ref{lemma:no_edges_many}, $H'$ does not contain a balanced independent set of size $s$.
Furthermore, by Lemma~\ref{lemma:small_degree}, $\Delta(H') \leq \Delta$.
Finally, note that $|V_i| - \gamma\,N = n$, completing the proof of Theorem~\ref{theo:independent}\ref{item:upper-bound}.

\subsection{Lower Bound}

For this section, we fix the following parameters:
\[p \defeq \left(\left(\frac{1 - \eps/4}{k-1}\right)\frac{\log D}{D}\right)^{1/(k-1)}, \quad \delta \defeq D^{-\left(\frac{1 - \eps/8}{k-1}\right)}.\]
Let $H = (V_1 \cup \cdots \cup V_k, E)$ be an $n$-balanced $k$-uniform $k$-partite hypergraph such that $|E| = Dn$.
Throughout the proof, we will assume $D$ is sufficiently large and $\eps$ is sufficiently small for all computations to hold.

We will prove Theorem~\ref{theo:independent} by analyzing the following procedure:
\begin{enumerate}[label=\ep{\textbf{Ind\arabic*}}]
    \item\label{item:color} For each $v \in V(H) \setminus V_k$, include $v \in I$ independently with probability $p $.
    \item For each $v \in V_k$, include $v \in I$ if and only if $\forall e \in E_H(v)$, $e - v \not\subseteq I$, where $e-v = e \setminus\set{v}$.
    \item Define $I' \subseteq I$ by removing vertices such that $|I'\cap V_i| = \min_{j \in [k]}|I \cap V_j|$ for each $i \in [k]$.
\end{enumerate}
Clearly the resulting set $I'$ is a balanced independent set.
We will show that there is an outcome $I$ such that
\[|I \cap V_i| \geq \left(\left(\frac{1-\eps}{k-1}\right)\frac{\log D}{D}\right)^{1/(k-1)}n,\]
for each $i \in [k]$.
First, let us show that $|I \cap V_j|$ is well concentrated for $j < k$.

\begin{Lemma}\label{lemma:Ij}
    $\P\left[\exists\, j \in [k-1],\, |I \cap V_j| \leq \left(\left(\frac{1-\eps}{k-1}\right)\frac{\log D}{D}\right)^{1/(k-1)}n\right] \leq 2k\exp\left(-\frac{\eps^2np}{300k^2}\right)$.
\end{Lemma}

\begin{proof}
    This follows from \hyperref[theo:chernoff]{Chernoff's inequality} and a union bound over $[k]$.
    In particular, we have the following:
    \begin{align*}
        \P\left[|I \cap V_j| \leq \left(\left(\frac{1-\eps}{k-1}\right)\frac{\log D}{D}\right)^{1/(k-1)}n\right] &= \P\left[|I \cap V_j| \leq \left(\frac{1-\eps}{1 - \eps/4}\right)^{1/(k-1)}\E[|I \cap V_j|]\right] \\
        &\leq \P\left[|I \cap V_j| \leq \left(1 - \eps/2\right)^{1/(k-1)}\E[|I \cap V_j|]\right] \\
        &\leq \P\left[|I \cap V_j| \leq \left(1 - \eps/(10k)\right)\E[|I \cap V_j|]\right] \\
        &\leq \P\left[||I \cap V_j| - \E[|I \cap V_j|]| \geq \eps/(10k)\,\E[|I \cap V_j|]\right] \\
        &\leq 2\exp\left(-\frac{\eps^2\E[|I \cap V_j|]}{300k^2}\right),
    \end{align*}
    as claimed.
\end{proof}

Now, let us consider $I\cap V_k$.

\begin{Lemma}\label{lemma:V_k_exp}
    $\E[|I \cap V_k|] \geq \delta\,n$.
\end{Lemma}

\begin{proof}
    Let us define the following random variables for each $v \in V_k$ and $e \in E_H(v)$:
    \[S_{v, e} \defeq \bbone\set{e-v \not\subseteq I}, \quad S_v \defeq \prod_{e \in E_{H}(v)}S_{v, e}.\]
    It follows that
    \[|I\cap V_k| = \sum_{v \in V_k}S_v.\]
    Note that
    \[\P[S_v = 1] = \P[\forall e \in E_H(v),\, S_{v, e} = 1], \quad \text{and} \quad  \P[S_{v, e} = 1] = 1 - p^{k-1}.\]
    In order to provide a lower bound on the former, we will apply \hyperref[theo:harris]{Harris's inequality}.
    For $v \in V(H) \setminus V_k$, define the event $I_v \defeq \set{v \in I}$.
    Let $\Gamma \defeq \set{I_v\,:\,v \in V(H) \setminus V_k}$, and let $S \subseteq \Gamma$ be the random events in $\Gamma$ that took place during step~\ref{item:color}.
    Note that $S$ is formed by including each event independently with probability $p$.
    Furthermore, $S$ determines the entire procedure, i.e., we can determine $I'$ from $S$.
    Consider the following family for $e \in E_H(v)$:
    \[\mathcal{A}_e \defeq \set{S' \subseteq \Gamma\,:\, \text{when } S=S' \text{ we have } S_{v, e} = 1}.\]
    Let $S_1 \in \mathcal{A}_e$, and $S_2 \subseteq S_1$.
    Then, $S_2 \in \mathcal{A}_e$ as well.
    In particular, $\mathcal{A}_e$ is a decreasing family of subsets of $\Gamma$.
    Hence, by \hyperref[theo:harris]{Harris's inequality}, we have
    \[\P[S_v = 1] \,\geq\, \prod_{e \in E_H(v)}\P[S_{v, e} = 1] \,=\, (1 - p^{k-1})^{\deg_H(v)},\]
    from where, we conclude by Jensen's inequality
    \[\E[|I\cap V_k|] \,=\, \sum_{v \in V_k}\P[S_v = 1] \,\geq\, \sum_{v \in V_k}(1 - p^{k-1})^{\deg_H(v)} \,\geq\, n(1 - p^{k-1})^D.\]
    Note the following:
    \begin{align*}
        (1 - p^{k-1})^D &\geq \exp\left(-\frac{p^{k-1}D}{(1 - \eps/20)}\right) \\
        &= \exp\left(-\frac{(1-\eps/4)\log D}{(k-1)(1 - \eps/20)}\right) \\
        &\geq D^{-\left(\frac{1 - \eps/8}{k-1}\right)},
    \end{align*}
    which completes the proof.
\end{proof}

It remains to bound the probability that $|I\cap V_k|$ is small.

\begin{Lemma}\label{lemma:V_k_conc_markov}
    $\P\left[|I\cap V_k| \leq \left(\left(\frac{1-\eps}{k-1}\right)\frac{\log D}{D}\right)^{1/(k-1)}n\right] \leq 1 - \delta/2$.
\end{Lemma}

\begin{proof}
    Note that $|I\cap V_k| \leq n$.
    From Lemma~\ref{lemma:V_k_exp} and the \hyperref[theo:markov]{Reverse Markov inequality}, we have
    \begin{align*}
        \P[|I\cap V_k| \leq \E[|I\cap V_k|/2]] &\leq \frac{n - \E[|I\cap V_k|]}{n - \E[|I\cap V_k|]/2} \\
        &= 1 - \frac{\E[|I\cap V_k|]/2}{n - \E[|I\cap V_k|]/2} \\
        &\leq 1 - \delta/2.
    \end{align*}
    Note the following:
    \begin{align*}
        \left(\left(\frac{1-\eps}{k-1}\right)\frac{\log D}{D}\right)^{1/(k-1)} \,\leq\, D^{-\left(\frac{1 - \eps/10}{k-1}\right)} \,\leq\, \delta/2,
    \end{align*}
    for $D$ large enough.
    In particular, we have
    \begin{align*}
        \P\left[|I\cap V_k| \leq \left(\left(\frac{1-\eps}{k-1}\right)\frac{\log D}{D}\right)^{1/(k-1)}n\right] &\leq \P\left[|I\cap V_k| \leq \delta n/2\right] \\
        &\leq \P[|I\cap V_k| \leq \E[|I\cap V_k|]/2] \\
        &\leq 1 - \delta/2,
    \end{align*}
    as desired.
\end{proof}

From Lemmas~\ref{lemma:Ij}, \ref{lemma:V_k_conc_markov}, and for $C$ large enough in the statement of Theorem~\ref{theo:independent}\ref{item:lower-bound}, we have
\begin{align*}
    \P\left[|I'| \leq \left(\left(\frac{1-\eps}{k-1}\right)\frac{\log D}{D}\right)^{1/(k-1)}n\right] &= \P\left[\exists j \in [k]\,,\,|I\cap V_j| \leq \left(\left(\frac{1-\eps}{k-1}\right)\frac{\log D}{D}\right)^{1/(k-1)}n\right] \\
    &\leq 2k\,\exp\left(-\frac{\eps^2np}{300k^2}\right) + 1 - \delta/2 \\
    &\leq 1 - \delta/4 \,<\, 1,
    % &< 1,
\end{align*}
for $n,\,D$ large enough.
Therefore, there exists an outcome where the complements of the events in Lemmas~\ref{lemma:Ij} and \ref{lemma:V_k_conc_markov} occur, completing the proof of Theorem~\ref{theo:independent}\ref{item:lower-bound}.

\section{Proof of Theorem~\ref{theo:color}}\label{sec:color}

For this section, we will fix the following parameters for $\gamma \defeq \eps/(2k^2)$:
\[q \defeq (1+\gamma/2)\left(\left(k - 1\right)\frac{\Delta}{\log \Delta}\right)^{1/(k-1)}, \quad \delta \defeq \exp\left(-\Delta^{\gamma/50}\right), \quad \omega \defeq\frac{1}{\Delta\,(\log\Delta)^{1/(2(k-1))}}.\]
To prove Theorem~\ref{theo:color}, we will describe a two stage coloring procedure.
We will first find a proper partial balanced coloring $\phi$ using $q$ colors.
We will show that the uncolored vertices induce a hypergraph of ``small'' maximum degree, from where we can complete the coloring by applying Lemma~\ref{lemma:rDelta}.
Before we describe the procedure, we make a few definitions regarding partial colorings $\phi\,:\,V(H) \to [q]$:
\begin{align*}
    \forall S \subseteq V(H),\, \phi(S) &\defeq \set{\phi(u)\,:\, u \in S}, \\
    \forall v \in V(H),\, L_\phi(v) &\defeq \set{c \in [q]\,:\, \forall e \in E_H(v),\, \phi(e-v) \neq \set{c}}.
\end{align*}
In particular, $L_\phi(v)$ contains colors which may be assigned to $v$ while ensuring the coloring remains proper.

We will construct $\phi$ through the following steps:
\begin{enumerate}[label=\ep{\textbf{Col\arabic*}}]
    \item\label{step:color} For each $v \in V(H) \setminus V_k$, independently assign $\phi(v) \in [q]$ uniformly at random.
    
    \item\label{step:proper}
    For each $v \in V_k$, if $L_\phi(v) \neq \0$, assign $\phi(v) \in L_\phi(v)$ uniformly at random.
    
    \item\label{step:uncolor_k} For each $c \in [q]$, let $V_i(c)$ be the set of vertices $v \in V_i$ such that $\phi(v) = c$, and let $n_c \in \N$ be such that $n_c \leq \min_{i \in [k]}|V_i(c)|$.
    Uncolor $|V_k(c)| - n_c$ vertices in $V_k(c)$ to obtain $V_k'(c)$.

    \item\label{step:uncolor_j} For each $c \in [q]$ and $1 \leq i < k$, uncolor $|V_i(c)| - |V_{k}'(c)|$ vertices in $V_i(c)$ to obtain $V_i'(c)$.
    
\end{enumerate}
A few remarks are required.
We note that step~\ref{step:proper} ensures $\phi$ is proper, and steps~\ref{step:uncolor_k}, \ref{step:uncolor_j} ensure that $\phi$ is balanced.
Under certain conditions and for an appropriate $n_c$, the choice of vertices to be removed during step~\ref{step:uncolor_k} can be arbitrary, however, we will describe later how to carefully choose vertices to uncolor during step~\ref{step:uncolor_j} such that the hypergraph $H_\phi$ induced by the uncolored vertices has small maximum degree and satisfies the conditions of Lemma~\ref{lemma:rDelta}.

Let us first consider step~\ref{step:color}.

\begin{Lemma}\label{lemma:color_k-1}
    $\P\left[\exists i \in [k-1], c \in [q], \, |V_i(c) - n/q| \geq 2\omega\,n/q\right] \leq 2kq\exp\left(-\frac{4\,\omega^2\,n}{3q}\right)$.
\end{Lemma}

\begin{proof}
    The proof follows by \hyperref[theo:chernoff]{Chernoff's inequality} and a union bound over $[k-1], \, [q]$.
\end{proof}

Let $U_k \subseteq V_k$ consist of the vertices $v$ such that $L_\phi(v) = \0$.
The next lemma will bound $\E[|U_k|]$.

\begin{Lemma}\label{lemma:exp_uncolor}
    $\E[|U_k|] \leq \delta\,n$.
\end{Lemma}

\begin{proof}
    For each $v \in V_k$ and $c \in [q]$, define the following events:
    \[S_{v, c} \defeq \bbone\set{c\notin L_\phi(v)}, \quad S_v \defeq \prod_{c\in [q]}S_{v, c}.\]
    It follows that
    \[\E[|U_k|] = \sum_{v \in V_k}\P[S_v = 1].\]
    We will compute an upper bound for $\P[S_v = 1]$ through a series of claims.
    Let us first consider the event $\set{S_{v, c} = 1}$.
    
    \begin{claim}\label{claim:pr_color_lost}
        $\P[S_{v, c} = 1] \leq 1 -\Delta^{-\frac{1-(k-1)\gamma/20}{k-1}}$.
    \end{claim}
    \begin{claimproof}
        We will lower bound $\P[S_{v, c} = 0]$ through \hyperref[theo:harris]{Harris's Inequality}.
        In order to do so, we define the following event for each $e \in E_H(v)$:
        \[S_e \defeq \bbone\set{\phi(e\setminus v) \neq \set{c}}.\]
        Let $\phi_u \defeq \set{\phi(u) = c}$, $\Gamma \defeq \set{\phi_u\,:\,u \in V(H)\setminus V_k}$, and let $S \subseteq \Gamma$ be the random events in $\Gamma$ that took place during step~\ref{step:color}.
        Note that $S$ is formed by including each event independently with probability $1/q$.
        Furthermore, $S_e = 0$ if and only if $\phi_u \in S$ for each $u \in e-v$.
        Consider the following families for $e \in E_H(v)$:
        \[\mathcal{A}_e \defeq \set{S' \subseteq \Gamma\,:\, \text{when } S=S' \text{ we have } S_e = 1}.\]
        Let $S_1 \in \mathcal{A}_e$, and $S_2 \subseteq S_1$.
        Then, $S_2 \in \mathcal{A}_e$ as well.
        In particular, $\mathcal{A}_e$ is a decreasing family of subsets of $\Gamma$.
        Hence, by \hyperref[theo:harris]{Harris's Inequality}, we have
        \[\P[S_{v, c} = 0] \geq \prod_{e \in E_H(v)}\P[S_{e} = 1] = (1 - 1/q^{k-1})^{\deg_H(v)} \geq \left(1 - 1/q^{k-1}\right)^\Delta.\]
        Note the following:
        \begin{align*}
            \left(1 - 1/q^{k-1}\right)^\Delta &\geq \exp\left(-\frac{\Delta}{(1 - \gamma/4)q^{k-1}}\right) \\
            &= \exp\left(-\frac{\log \Delta}{(1+\gamma/2)^{k-1}(1-\gamma/4)(k-1)}\right) \\
            &\geq \exp\left(-\frac{\log \Delta}{(1+(k-1)\gamma/2)(1-\gamma/4)(k-1)}\right) \\
            &\geq \exp\left(-\frac{\log \Delta}{(1+(k-1)\gamma/10)(k-1)}\right) \\
            &\geq \Delta^{-\frac{1-(k-1)\gamma/20}{k-1}}.
        \end{align*}
        In particular, we can conclude
        \[\P[S_{v, c} = 1] \leq 1 -\Delta^{-\frac{1-(k-1)\gamma/20}{k-1}},\]
        as desired.
    \end{claimproof}
    
    In the next claim, we will show that the events $\set{S_{v, c} = 1}$ for $c \in [q]$ are negatively correlated.

    \begin{claim}\label{claim:neg_corr}
        For every $I \subseteq [q]$, we have 
        \[\P\left[\forall c \in I,\, S_{v, c} = 1\right] \leq \prod_{c \in I}\P\left[S_{v, c} = 1\right].\]
    \end{claim}
    \begin{claimproof}
        We proceed by induction.
        The claim is trivial for $|I| \leq 1$.
        Suppose it holds for all $I$ such that $|I| = \ell$.
        Consider such a set $I$ and a color $c' \notin I$.
        Let $X \defeq N_H(v)$.
        Form $S \subseteq X$ by including each vertex in $X$ independently with probability $1/q$.
        For each $u \in X \setminus S$, let $\psi(u) \in [q] \setminus \set{c'}$ be chosen uniformly at random.
        The following functions will assist with our proofs:
        \begin{align*}
            f(S) &\defeq \bbone\set{\exists e \in E_H(v)\text{ such that } e - v \subseteq S}, \\
            g(S) &\defeq \P\left[\forall c \in I,\, \exists e \in E_H(v) \text{ such that } e-v \subseteq X\setminus S \text{ and } \psi(e-v) = \set{c}\right].
        \end{align*}
        Note the following:
        \begin{itemize}
            \item $\P[S_{v,c'} = 1] = \E[f(S)]$,
            \item $\P[\forall c \in I,\, S_{v,c} = 1] = \E[g(S)]$, and 
            \item $\P[\forall c \in I\cup \set{c'},\, S_{v,c} = 1] = \E[f(S)g(S)]$,
        \end{itemize}
        where the expectation is taken over $S$.
        Furthermore, $f$ is an increasing function and $g$ is a decreasing function with respect to $2^X$.
        By the \hyperref[theo:harris]{FKG Inequality}, we conclude that
        \begin{align*}
            \P[\forall c \in I\cup \set{c'},\, S_{v,c} = 1] &= \E[f(S)g(S)] \\
            &\leq \E[f(S)]\E[g(S)] \\
            &\leq \P[S_{v,c'} = 1]\P[\forall c \in I,\, S_{v,c} = 1].
        \end{align*}
        The claim now follows by the induction hypothesis.
    \end{claimproof}
    By Claims~\ref{claim:pr_color_lost} and \ref{claim:neg_corr}, we have
    \begin{align*}
        \P[S_v = 1] \leq \prod_{c \in [q]}\P[S_{v, c} = 1] &\leq \left(1 - \Delta^{-\frac{1-(k-1)\gamma/20}{k-1}}\right)^q \\
        &\leq \exp\left(-\frac{q}{\Delta^{\frac{1-(k-1)\gamma/20}{k-1}}}\right) \\
        &\leq \exp\left(-\left(\frac{\Delta^{(k-1)\gamma/20}}{\log \Delta}\right)^{1/(k-1)}\right).
    \end{align*}
    As $\log \Delta \leq \Delta^{\gamma/25}$ for $\Delta$ large enough, this completes the proof of Lemma~\ref{lemma:exp_uncolor}.
\end{proof}

The following lemma follows trivially from \hyperref[theo:markov]{Markov's inequality} and Lemma~\ref{lemma:exp_uncolor}.

\begin{Lemma}\label{lemma:conc_uncolor}
    $\P\left[|U_k| \geq 2\delta\,n\right] \leq \frac{1}{2}$.
\end{Lemma}

Let us now consider the sets $V_k(c)$ for each $c \in [q]$.
Note the following:
\[\sum_{c \in [q]}|V_k(c)| = n - |U_k|.\]
By symmetry, $|V_k(c)|$ has the same distribution for each $c \in [q]$.
It follows from Lemma~\ref{lemma:exp_uncolor} and by linearity of expectation that:
\begin{align}\label{eqn:exp_V_k_c}
    \frac{n}{q}(1 - \delta) \,\leq\, \E[|V_k(c)|] \,\leq\, \frac{n}{q}.
\end{align}

\begin{Lemma}\label{lemma:k_color_classes}
    $\P\left[\exists c \in [q],\, ||V_k(c)| - \E[|V_k(c)|]| \geq \omega\,n/q \right] \leq 4q\exp\left(-\frac{\omega^2\,n}{64\,q\,\Delta^3\,k}\right)$.
\end{Lemma}

\begin{proof}
    We will use \hyperref[theo:Talagrand]{Talagrand's inequality} to prove concentration for a fixed $c \in [q]$.
    The claim then follows by a union bound.
    For each $u \in V(H) \setminus V_k$, let $T_u$ be the random variable denoting $\phi(u)$ in step~\ref{step:color}.
    Furthermore, for each $v \in V_k$, let $T_v$ be i.i.d.\ uniform random variables on $[0,1]$.
    If $L_\phi(v) \neq \0$, we let $\phi(v)$ be the $\lceil T_v\,|L_\phi(v)|\rceil$-th element in $L_\phi(v)$.
    (It can be verified that this is equivalent to picking an element uniformly at random from $L_\phi(v)$.)
    Changing the outcome of a single $T_u$ can affect $|V_k(c)|$ by at most $\Delta$.
    Furthermore, if $|V_k(c)| \geq s$ for some $s$, we can certify this by the outcomes of at most $(\Delta\,(k-1) + 1)s$ trials $T_u$.
    In particular, for any $s$ vertices $v \in V_k(c)$, we consider the trial $T_v$ and the trials $T_u$ for $u \in N_H(v)$.
    Since $\E[|V_k(c)|] \leq n/q$ as a result of \eqref{eqn:exp_V_k_c}, we can now apply \hyperref[theo:Talagrand]{Talagrand's inequality} with $r = (\Delta\,(k-1) + 1)$, and $\gamma = \Delta$ to get:
    \begin{align*}
        \P[||V_k(c)| - \E[|V_k(c)|]| \geq \omega\,n/q] &\leq \P\left[||V_k(c)| - \E[|V_k(c)|]| \geq \omega\,n/(2q) + 20\gamma\sqrt{r\E[|V_k(c)|]} + 64\gamma^2r\right] \\
        &\leq 4\exp\left(-\frac{\omega^2\,n^2}{32\Delta^3\,k\,q^2(\E[|V_k(c)|] + \omega\,n/2q)}\right) \\
        &\leq 4\exp\left(-\frac{\omega^2\,n}{64\,q\,\Delta^3\,k}\right),
    \end{align*}
    where the inequalities follow for $n$ large enough.
\end{proof}

Recall the bound on $n$ in the statement of Theorem~\ref{theo:color}.
Provided $C$ is large enough, the complements of Lemmas~\ref{lemma:color_k-1},~\ref{lemma:conc_uncolor}, and~\ref{lemma:k_color_classes} occur with probability at least
\[\frac{1}{2} - 2kq\exp\left(-\frac{4\,\omega^2\,n}{3q}\right) - 4q\exp\left(-\frac{\omega^2\,n}{64\,q\,\Delta^3\,k}\right) > 0.\]
Assuming such an outcome, we have the following as a result of \eqref{eqn:exp_V_k_c} and since $\omega \gg \delta$:
\begin{align}\label{item:classes}
    |V_j(c)| \in \frac{n}{q}(1 \pm 2\omega), \quad  \text{for each $j \in [k]$ and $c \in [q]$}. 
\end{align}

When forming $V_k'(c)$ at step~\ref{step:uncolor_k}, we let $n_c \defeq (1 - 2\omega)n/q$ for each $c \in [q]$ and arbitrarily uncolor at most $4\omega\,n/q$ vertices and add them to $U_k$ to define $U_k'$.
It follows that
\[|U_k'| \,=\, |U_k| + \sum_{c \in [q]}|V_k(c) \setminus V_k'(c)| \,=\, n - \sum_{c \in [q]}|V_k'(c)| \,=\, 2\omega\,n.\]
Let us define the following parameter:
\[\Tilde{\Delta} \defeq \frac{\gamma}{2k}\left((k-1)\frac{\Delta}{\log\Delta}\right)^{1/k-1} = \frac{\gamma\,q}{(2+\gamma)k}.\]
Note that $\tilde \Delta \ll \omega\,n = |U_k'|/2$.
The goal is to construct $V_j'(c)$ at step~\ref{step:uncolor_j} such that the uncolored vertices induce a hypergraph $H_\phi$ of maximum degree \textbf{strictly less than} $\Tilde{\Delta}$.
Since
\[q + k\Tilde{\Delta} \,=\, (1+\gamma)\left((k-1)\frac{\Delta}{\log\Delta}\right)^{1/k-1} \,\leq\, \left((k-1 + \eps)\frac{\Delta}{\log\Delta}\right)^{1/k-1},\]
we would then be able to complete the proof of Theorem~\ref{theo:color} by applying Lemma~\ref{lemma:rDelta} to $H_\phi$.

Let $j \in [k]$ and $c \in [q]$ be arbitrary.
Define the following set:
\[B_j(c) \defeq \set{u \in V_j(c)\,:\, |\set{e \in E_H(u)\,:\, e\cap U_k' \neq \0}| \geq \Tilde{\Delta}}.\]
In particular, $B_j(c)$ consists of the ``bad'' vertices in $V_j(c)$, i.e., the vertices that may have large degree in $H_\phi$.
It is enough to show that we can define $V_j'(c)$ by only uncoloring ``good'' vertices.
To this end, we note the following:
\begin{align*}
    |B_j(c)|\,\Tilde{\Delta} \,\leq\, \Delta|U_k'| \,\implies\, |B_j(c)| \,\leq\, \frac{\Delta}{\Tilde{\Delta}}|U_k'| \,\leq\, \frac{5k\omega\,\Delta\,(2+\gamma)}{\gamma}\,\frac{n}{q}.
\end{align*}
In particular, we have
\[|V_j(c) \setminus B_j(c)| \,\geq\, \frac{n}{q}\left(1 - 2\omega -\frac{6k\omega\,\Delta\,(2+\gamma)}{\gamma} \right) \,\gg\, 4\omega\,\frac{n}{q}.\]
As a result of \eqref{item:classes}, we need to remove at most $4\omega\,n/q$ vertices from $V_j(c)$, and as a result of the above inequality, these can all be chosen to be ``good'' vertices.

\subsection*{Acknowledgements}

We thank the anonymous referees for their helpful comments.

\vspace{0.1in}
\printbibliography

\end{document}